%% file: homotopytype.tex
\documentclass[10pt,a4paper,reqno]{amsart}
\input{defs}

\newcommand{\ppp}[1]{\pp^{#1}}
\def\bump{\nu}
\def\bumpp{\nu}

\def\fib{\mathrm{fib}}
\def\P{\Psi_d}
\def\Pr{{\widetilde{\Psi}_d}}
\def\Pg{\Pr}
\def\ls{{\widetilde{\Psi}}}
\def\gs{\Psi}
\def\ms{{\Psi(-,X)}}
\def\Pgcero{{\widetilde{\Psi}_0}}
\def\Prcero{\widetilde{\Psi}_0}

\def\pp{\Pr(\bR^n)}
\def\qq{\Pr(B^n)}

\def\L{\cL_d}
\def\Lr{\widetilde{\cL}_d}
\def\Lg{\widetilde{\cL}_d}

\def\kk{\L(\bR^n)}
\def\ll{\Lr(\bR^n)}

\def\Gr{\mathrm{Gr}}
\def\diam{\mathrm{diam}}
\def\Gauss{\mathrm{Gauss}}
\def\gauss{\mathrm{Gauss}}

\def\CL{\mathrm{CL}}

\title[The homotopy type of the space of submanifolds]{Homotopy types of spaces of submanifolds of $\bR^n$}
\author{Federico Cantero Mor\'an}
\thanks{The author was funded by Michael Weiss’ Humboldt professor grant and the Belgian Interuniversity
Attraction Pole P07/18. He was partially supported by the Spanish Ministry of Economy
and Competitiveness under grant MTM2013-42178-P.}% and by the Generalitat de Catalunya as member of the team 2009 SGR 119.}
\email{\texttt{fcant{\_}01@uni-muenster.de}}
\address{{\normalfont IRMP, Chemin du cyclotron, 2 1348 Louvain-la-nueuve, Belgium}}
\date{}


\begin{document}
\begin{abstract} We compute the homotopy type of the space of proper $d$-dimensional submanifolds of $\bR^n$ with a smooth version of the Fell topology. Our methods allow us to compute the homotopy type of the space of submanifolds with summable labels too, and to give a new proof of the Galatius--Randal-Williams theorem on the homotopy type of their space of submanifolds.
\end{abstract}
\maketitle

\section{Introduction}

In \cite{GMTW}, the classifying space of the $d$-dimensional cobordism category was found to be homotopy equivalent to a delooping of the infinite loop space associated to the Thom spectrum $\mathbf{MTO}(d)$, whose $n$-th space is the Thom space of the affine Grassmannian $\gamma_{d,n}^\perp$ of $d$-planes in $\bR^n$, seen as a vector bundle over the linear Grassmannian $\Gr_d(\bR^n)$. This was proven again with different methods by Galatius and Randal-Williams \cite{GR-W}, who introduced the space $\P(\bR^n)$ of submanifolds of $\bR^n$. A crucial step in their proof was proving that the inclusion 
\begin{equation}\label{eq:010}\Th(\gamma_{d,n}^\perp)\hookrightarrow \P(\bR^n)\end{equation}
is a weak homotopy equivalence.

$\P(\bR^n)$ is a topological space whose underlying set $\psi_d(\bR^n)$ is the subset of the power set of $\bR^n$ consisting of (possibly empty, possibly non-compact) proper $d$-submanifolds of $\bR^n$ without boundary. Here, a \emph{proper subset of $\bR^n$} is a subset of $\bR^n$ whose intersection with any compact subset is compact. A subset of $\bR^n$ is proper if and only if it is closed. 

In this paper, we study a different topology of the set $\psi_d(\bR^n)$ and, via a counterpart to \eqref{eq:010}, show that $\psi_d(\bR^n)$ with this alternative topology is weakly contractible. We also study the set $\psi_d(\bR^n;X)$ of proper $d$-dimensional submanifolds of $\bR^n$ with summable labels on an abelian topological monoid $X$. We start by defining the new topology on $\psi_d(\bR^n)$, the differential Fell topology, and explain how it relates to the topology on $\P(\bR^n)$. The definition of the topology on $\psi_d(\bR^n;X)$ is postponed until Definition \ref{df:ms}.

%
%Before describing the topology in $\P(\bR^n)$, let us discuss what natural topologies does $\psi_d(\bR^n)$ carry:

The set of closed subsets of a topological space $X$ can be endowed with several topologies, of which the Fell topology \cite{Fell} is the most convenient for us. To describe it define, for each subset $U\subset X$, 
\[U^-= \{A\in \CL(X)\mid A\cap U\neq \emptyset\}, \quad U^+=\{A\in \CL(X)\mid A\subset U\}.\]
Then the Fell topology has as subbasis the collection of all subsets $U^-$ with $U$ an open subset of $X$ and $U^+$ with $U$ the complement of a compact subset of $X$. The reader more familiar with topologies in function spaces may gain some intuition by knowing that if $X$ is metrizable, then there exists a metric $d$ on $X$ for which the function
\begin{equation}\label{eq:17}\CL(X)\lra \map(X,\bR)\end{equation}
given by $A\mapsto d(A,-)$ is an embedding if the right hand-side is endowed with the compact-open topology \cite[Theorem~2(d)]{Beer:Fell}.

The subset $\psi_d(\bR^n)\subset \CL(\bR^n)$ can be endowed with the subspace topology, but this of course does not take into account the smooth structure of the submanifolds. This is addressed by considering $\psi_d(\bR^n)$ not as a subset of $\CL(\bR^n)$, but as a subset
 \[\psi_d(\bR^n)\subset \CL(\bR^n\times \Gr_d(\bR^n)),\] 
using the Gauss map that sends a submanifold $W$ to $\{(x,T_xW)\mid x\in W\}$.

 We denote by $\pp$ the set $\psi_d(\bR^n)$ endowed with the subspace topology of the latter inclusion, and we refer to it as the \emph{space of merging submanifolds} of $\bR^n$.\label{merging}

This topology is still coarser than the topology in $\P(\bR^n)$, but it is very close to it \cite{cantero:metric}: A sequence of compact submanifolds in $\Pr(\bR^n)$ converging to a compact submanifold $W$ eventually takes values in covering spaces of $W$. If we refine the topology $\Pr(\bR^n)$ imposing these covering spaces to be single-sheeted (i.e., diffeomorphisms), then we arrive to the topology $\P(\bR^n)$ defined by Galatius and Randal-Williams. \mnote{All the topologies described so far give rise to topological sheaves over $\bR^n$.}

The following is the main theorem of this paper. %It improves \eqref{eq:010} to a strong deformation retraction and generalises it to the other two spaces of submanifolds.
Let $\kk$ and $\ll$ be the subspaces of $\P(\bR^n)$ and $\pp$, consisting only on (possibly empty) unions of parallel affine planes together with the empty set. Let $\Lg(\bR^n;X)$ be the subspace of $\Pg(\bR^n;X)$ of (possibly empty) unions of parallel affine planes and locally constant labels (in the notation of Definition \ref{df:ms}, $\alpha$ is locally constant).
\begin{theorem} The inclusions 
\begin{align*}
 \L(\bR^n)&\hookrightarrow \P(\bR^n)\\
 \ll&\hookrightarrow \pp\\
 \Lg(\bR^n;X)&\hookrightarrow \Pg(\bR^n;X)
\end{align*}
are part of a strong deformation retraction. In addition, the inclusion $\Th(\gamma_d^\perp(\bR^n))\hookrightarrow \L(\bR^n)$ is part of a strong deformation retraction and $\ll$ is weakly contractible.
\end{theorem}

Though the main contribution of this theorem is the computation of the homotopy type of $\pp$ and $\Pg(\bR^n;X)$, the part referring to $\P(\bR^n)$ has its own interest: Galatius and Randal-Williams showed that \eqref{eq:010} is a weak homotopy equivalence, and a formal corollary of our main theorem improves their result as follows: 

\begin{cor0} The inclusion \eqref{eq:010} is a strong deformation retraction, so in particular $\P(\bR^n)$ has the homotopy type of a CW-complex.
\end{cor0}

The assignments $\P(-)$, $\Pr(-)$ and $\Pg(-;X)$ define sheaves on the site of manifolds and open embeddings (see Section \ref{s:sheaves} and \cite{GR-W}). In \cite{R-WEmbedded}, Randal-Williams proved the remarkable property that $\P(-)$ is a microflexible sheaf. In Section \ref{s:4} we show that $\Pr(-)$ and $\Pg(-)$ are not microflexible, so the results of \cite[\S 3-6]{R-WEmbedded} on $\P(-)$ (which are based on \cite{galatius-2006}) do not generalize to $\Pr(-)$ and $\Pg(-)$.

\subsection*{Acknowledgements} This paper has benefited from many conversations with Abd\'o Roig-Maranges, as well as from comments from the anonymous referee, Oscar Randal-Williams and Martin Palmer. %It has been significantly improved thanks to the work of the anonymous referee. 
The paper by Karcher in the references was kindly pointed out to me by Igor Belegradek at mathoverflow.

%When $d=0$, the space $\pp$ is known as the Ran space on $\bR^n$. %Therefore these spaces could be regarded as a possible higher-dimensional generalisation of the Ran space.
%\mnote{fc:although a more natural generalisation would be to take the space of images of continuous maps from $d$-manifolds. This is also a weakly contractible $h$-space.}

\section{Spaces of submanifolds and an overview of the proof}

In this section we give a precise definition of the spaces $\P(U)$ and $\Pr(U)$ (\cite[\S 2]{GR-W}, \cite{cantero:metric}) and we give a some comments on the proof. For the sake of clarity, we will give a $C^1$-version of this topology as in \cite{Bokstedt-Madsen}. We start setting up some conventions:

We denote by $d_0$ the Euclidean distance in $\bR^n$, by $d_1$ the distance on $\Gr_d(\bR^n)$ given by
\[d_1(L,L') = \max_{v\in S(L)}\min_{w\in S(L')}\{\mathrm{angle}(v,w)\} = \max_{w\in S(L')}\min_{v\in S(L)}\{\mathrm{angle}(v,w)\},\]%explanation http://mathoverflow.net/questions/48118/a-metric-for-grassmannians
where $S(L)$ is the unit sphere in $L$, and by $d = d_0+d_1$ the distance in $\bR^n\times \Gr_d(\bR^n)$. If $f\colon L\to L'$ is a linear operator, then 
\[\|f\| = \max_{v\in S(L)}\{\|f(v)\|\}.\]

If $W\in \psi(U)$, we denote by $p\colon NW\to W$ the projection, which is covered by two bundle maps: its differential $Dp\colon T(NW)\to TW$ and the canonical bundle isomorphism
\[\xymatrix{
TNW \ar[r]^-\alpha\ar[d] & TW\oplus NW\ar[d] \\
NW\ar[r]^p & W.
}\]
The differential $Dp$ is the composition of $\alpha$ with the projection onto $TW$, and we write $\tau$ for the composition of $\alpha$ with the projection onto $NW$. 

Recall that there is a function $\epsilon\colon W\to (0,\infty)$ such that the restriction
\[\exp^\epsilon_W\colon N^\epsilon W\longrightarrow U\]
 of the exponential map $\exp_W\colon NW\rightarrow U$ to the subspace $N^\epsilon W$ of vectors $v$ of length at most $\epsilon(p(v))$ is an embedding. In addition, the restriction of this map to each fibre 
\[\exp_{W,x}^\epsilon\colon N_x^\epsilon W\to U\]
is a radial isometry. We denote by $z\colon W\to NW$ the zero section, by $T^\epsilon$ the image of $N^\epsilon W$, and by $\pi$ the composition of $(\exp^{\epsilon}_W)^{-1}$ and $p$. Summarizing:
\[\xymatrix{
NW \ar @{}[r]|{\supset}\ar[d]_p & N^\epsilon W\ar[d]_{\exp_W^\epsilon}^\cong & \\
W\ar @{}[r]|{\subset} \ar @/^/[ur]^{z} &T^\epsilon\ar @{}[r]|{\subset} \ar @/^/[l]^{\pi} & U
}\]

%\begin{df} The topology in $\Psi(\bR^n)$ is defined as follows: If $W\neq \emptyset$, every compact subset $K\subset U$ and every neighbourhood $A$ of the zero section defines a basic neighbourhood $(K,A)^\gs$ of $W$; a submanifold $W'$ belongs to $(K,A)^\gs$ if there is a section $f\in A$ of the normal bundle $NW\to W$ such that $f(W)\subset A$ and 
%\[\exp_W(f(W))\cap K = W'\cap K.\]
%If $W=\emptyset$, then every compact subset $K\subset \bR^n$ defines a basic neighbourhood $(K)^\gs$ of $\emptyset$; and a submanifold $W'$ belongs to $(K)^\gs$ if $W'\cap K=\emptyset$.
%\end{df}

\begin{df} The space $\P(U)$ has as underlying set the collection of all proper $d$-dimensional submanifolds of $U$, together with the empty submanifold. Its topology is given by the following neighbourhood basis of any proper submanifold $W$: every compact subset $K\subset U$ and every $\epsilon>0$ define a basic neighbourhood $(K,\epsilon)^\gs$ of $W$; a submanifold $W'$ belongs to $(K,\epsilon)^\gs$ if there is a section $f$ of the normal bundle $NW\to W$ such that 
\begin{enumerate}
\item $\exp_W(f(W))\cap K = W'\cap K$ and 
\item\label{bla} $\|f(x)\| + \|\tau\circ D(f)(x)\| < \epsilon$ for all $x\in W$ such that $\exp_W\circ f(x)\in W'\cap K$.
\end{enumerate}
%\item If $W=\emptyset$, then every compact subset $K\subset \bR^n$ defines a basic neighbourhood $(K)^\gs$ of $\emptyset$; and a submanifold $W'$ belongs to $(K)^\gs$ if $W'\cap K=\emptyset$.
\end{df}
We now define the topology in $\pp$, the only difference being that instead of requiring $W'\cap K$ to be the image of a global section of $NW$, we only ask it to be the union of images of local sections of $NW$ whose domains cover $W$. The fact that this is the topology in the space of merging submanifolds (p. \pageref{merging}) is shown in \cite[Theorem~1]{cantero:metric}.%\cite[\S 4.3]{FC-thesis}.

\begin{df}\label{df:ls} The space $\Pr(U)$ has the same underlying set as $\P(U)$, with neighbourhood basis of a proper submanifold $W$: every compact subset $K\subset U$ and every $\epsilon>0$ define a basic neighbourhood $(K,\epsilon)^\ls$ of $W$; a submanifold $W'$ belongs to $(K,\epsilon)^\ls$ if there is a subset $Q\subset NW$ such that the composite $Q\subset NW\to W$ is a smooth covering map onto $W\cap K$ and
%for each $y\in W'\cap K$,
%(resp. each $z\in W\cap K$),
% there is a local section $f_y$ 
%(resp. $f_z$) 
%of the normal bundle $NW\to W$ such that
\begin{enumerate}
\item\label{qui2} $\exp_W(Q)\cap K = W'\cap K$,
%\item $\exp_W(f_z(W))\cap K \subset W'\cap K$,
\item\label{bla2} $\|f(x)\| + \|\tau\circ D(f)(x)\| < \epsilon$ for each local section $f$ of the covering map $q$.
%\item $\|f_z(x)\| + \|D(f_y-\Id)(x)\| < \epsilon$ for all $x$ in the domain of $f_z$.
%\item \label{qua2} the union of the domains of $f_y$ cover $W\cap K$.
\end{enumerate}

%\item If $W=\emptyset$, then every compact subset $K\subset \bR^n$ defines a basic neighbourhood $(K)^\ls$ of $\emptyset$; and a submanifold $W'$ belongs to $(K)^\ls$ if $W'\cap K=\emptyset$.
%\end{itemize}
\end{df}

Let $U\subset \bR^n$ be an open subset, let $X$ be an abelian topological monoid with unit and let $\psi_d(U;X)$ be the set of pairs $(W,\alpha)$ with $W\in \psi_d(U)$ and $\alpha\colon W\to X$ a continuous map.

\begin{df}\label{df:ms} The space $\Pg(U;X)$ has underlyting set $\psi_d(U;X)$, with the following neighbourhood basis of a pair $(W,\alpha)$: every compact subset $K\subset U$, every $\epsilon>0$ and every open neighbourhood $A$ of $\alpha_{|W\cap K}$ in $\map(W\cap K,X)$ define a basic neighbourhood $(K,\epsilon,A)^{\ms}$ of $W$; a submanifold $W'$ belongs to $(K,\epsilon,A)^\ms$ if there is a subset $Q\subset NW$ such that the composite $Q\subset NW\to W$ is a smooth covering map onto $W\cap K$ and
%for each $y\in W'\cap K$,
%(resp. each $z\in W\cap K$),
% there is a local section $f_y$ 
%(resp. $f_z$) 
%of the normal bundle $NW\to W$ such that
\begin{enumerate}
\item $\exp_W(Q)\cap K = W'\cap K$,
%\item $\exp_W(f_z(W))\cap K \subset W'\cap K$,
\item $\|f(x)\| + \|\tau\circ D(f)(x)\| < \epsilon$ for each local section $f$ of the covering map $q$.
%\item $\|f_z(x)\| + \|D(f_y-\Id)(x)\| < \epsilon$ for all $x$ in the domain of $f_z$.
%\item \label{qua2} the union of the domains of $f_y$ cover $W\cap K$.
\item The map $\beta\colon W\cap K\to X$, given by $\beta(x) = \sum_{y\in q^{-1}(x)}\alpha'(y)$, belongs to $A$.
\end{enumerate}

%\item If $W=\emptyset$, then every compact subset $K\subset \bR^n$ defines a basic neighbourhood $(K)^\ls$ of $\emptyset$; and a submanifold $W'$ belongs to $(K)^\ls$ if $W'\cap K=\emptyset$.
%\end{itemize}
\end{df}

\begin{remark} \label{remark:100}
\begin{enumerate}
%\item One can also define the topology in $\Pr(U)$ following the definition of the topology in $\Psi(U)$, but replacing the normal bundle of $W$ by the fibrewise Ran space on the normal bundle of $W$ and adding the condition that the cardinality of the image of the section to be locally constant on $W$.
\item\label{item:100} If $W'\in (K,\epsilon)^\gs$ (respectively, $W'\in (K,\epsilon)^\ls$) and $f$ is a global (resp. local) section of the corresonding $q\colon Q\to W$, then \cite[Lemma~3.1]{cantero:metric}
%\begin{itemize}
%\item[(\ref{bla})'] 
\begin{equation}\label{eq:5}
d_0(x,f(x)) + d_1(T_xW,T_{f(x)}W')<\epsilon
\end{equation}
for all $x\in W$ such that $\exp_W\circ f(x)\in W'\cap K$.
%\item[(\ref{bla2})']
%\item $d(x,f_y(x)) + d(T_xW,T_{f_y(x)}W')<\epsilon$ for all $x$ in the domain of $f_y$.
%\end{itemize}
\item Condition (\ref{bla}) in both definitions says that $f$ is $\epsilon$-close to the zero section in the $C^1$-topology. One can instead impose that condition in the $C^\infty$-topology. This would give, in the case of $\Psi_d(U)$, the actual definition given by Galatius and Randal-Williams. This change does not change the weak homotopy type of $\P(\bR^n)$ and $\pp$.
\item When $d=0$, the subspace of $\Psi_0(\bR^n)$ consisting of $0$-submanifolds contained in the unit ball, is the unordered configuration space on the unit ball, whereas that subspace in $\Prcero(\bR^n)$ is the Ran space of the unit ball. Think now of $S^n$ as the one-point compactification of $\bR^n$, with the north pole as the point at infinity. The space $\Prcero(\bR^n)$ is the subspace of the Ran space of $S^n$ consisting on configurations that contain the north pole. The space $\Pgcero(\bR^n;\bN\setminus\{0\})$ is the infinite symmetric product of $S^n$ with the north pole as the basepoint.
\item \label{remit:10} There is a natural identification of $\Psi(U)$ with the subspace of $\Pr(U;\bZ)$ of pairs of the form $(W,1)$, and a natural identification of $\Pr(U)$ with the subspace of $\Pr(U;\bZ)$ of pairs of the form $(W,0)$. 
\end{enumerate}
\end{remark}

We will give the proof of the first part of the main theorem only for the inclusion $\ll\hookrightarrow \Pr(\bR^n)$, as the proof for the other two inclusions is exactly the same. Here is an overview of the proof: Let $\ppp{\delta}\subset \pp$ be the subspace of those submanifolds $W$ for which the diameter of the image of the Gauss map
\[\Gauss\colon W\to \Gr_d(\bR^n)\]
is at most $\delta$. Then, in Proposition \ref{prop:2} we will prove that when $\delta>0$, the inclusion $\ppp{\delta}\subset \pp$ is a homotopy equivalence that fixes $\ll = \ppp{0}\subset \ppp{\delta}$. Then, in Proposition \ref{prop:4} we will assume that $\delta$ is small, and we will continuously assign to each submanifold $W\in \ppp{\delta}$, a $d$-plane $\mu(W)$ such that $\mu(W)^\perp$ is transverse to $W$. In Proposition \ref{prop:3} we will construct a deformation retraction $\ppp{\delta}\to \ll$ that stretches out $W$ in the direction $\mu(W)$. Finally, in Propositions \ref{prop:5} and \ref{prop:6} we find the homotopy type of $\cL(\bR^n)$ and $\ll$.

%\begin{theorem} Let $\ll$ be the subspace of $\pp$, consisting only on unions of parallel affine planes, together with the empty set. The inclusion $i\colon \ll\hookrightarrow \pp$ is a homotopy equivalence and the space $\ll$ is weakly contractible.
%\end{theorem}
%
%The strategy to construct a homotopy inverse to $i$ will be the following: We continuously change $W$ through a homotopy $H_t$ that ends in a submanifold $V = H_0(W)$ whose tangent planes are close to each other, seen as points in $\Gr_d(\bR^n)$. To the submanifold $V$ we can continuously assign a plane $\mu(V)$ obtained by averaging all tangent planes of $V$. The plane $\mu(V)$ will have another property: every tangent plane to $V$ is transverse to $\mu(V)^\perp$. Finally, we stretch out $V$ in the direction $\mu(V)$, obtaining several parallel copies of $\mu(V)$, one for each point in $\mu(V)^\perp\cap V$. 

\begin{remark} The first part of the proof (Section \ref{section:almost}) can be significatively shortened if one is only interested in proving that the inclusion $\ll\subset \qq$ is a weak homotopy equivalence (see Remark \ref{remark:23}).
\end{remark}

\section{A technical lemma and an assumption}\mnote{fc:This section is complete and under revision}
\subsection{The action of a space of embeddings on the space of merging submanifolds}\label{s:sheaves}
The next lemma can be proven as in \cite[\S 2.2]{GR-W}, making the appropriate modifications. Instead, we will take advantage of knowing that the topology of $\Pr(U)$ is induced by the Fell topology.

\begin{lemma}\label{lemma:31} Let $U,V$ be open subsets of $\bR^n$ and let $\Emb(U,V)$ be the space of embeddings of $U$ into $V$ with the $C^1$ compact-open topology. Then the map
\begin{equation}\label{lemma:cont}\Emb(U,V)\times \Pr(V)\lra \Pr(U)\end{equation}
given by sending a pair $(f,W)$ to $f^{-1}(W)$ is continuous.
\end{lemma}
\begin{proof}
Let first $X,Y$ be metric spaces. The function 
\[\map(X,\bR)\to \CL(X)\]
 given by $f\mapsto f^{-1}(0)$ is easily seen to be continuous if $X$ is locally path-connected and in that case, it is a retraction of the map \eqref{eq:17}, hence a quotient map. Therefore, in the commutative diagram
\[\xymatrix{
\Emb(X,Y)\times \map(Y,\bR)\ar[r]\ar[d] & \map(X,\bR) \ar[d] \\
\Emb(X,Y)\times \CL(Y)\ar[r] & \CL(X)  
}\]
where the upper horizontal arrow is composition and the lower horizontal arrow is evaluation, the latter is continuous by the universal property of quotient maps. 

Now, let $U,V$ be open subsets of $\bR^n$. Then the map in the statement fits in the commutative diagram
\[\xymatrix{
\Emb(U,V)\times \Pr(\bR^n)\ar[r]\ar[d] & \Pr(\bR^n) \ar[d] \\
\Emb(U\times \Gr_d(\bR^n),V\times\Gr_d(\bR^n))\times \CL(V\times \Gr_d(\bR^n))\ar[r] & \CL(U\times \Gr_d(\bR^n)).
}\]
The lower horizontal map has been shown to be continuous, and the left vertical map is easily seen to be continuous too. Therefore the upper horizontal map is continuous because the right vertical map is an embedding.
\end{proof}

\begin{lemma}\label{lemma:32} Let $U,V$ be open subsets of $\bR^n$ and let $\Emb(U,V)$ be the space of embeddings of $U$ into $V$ with the $C^1$ compact-open topology. Then the map
\begin{equation}\Emb(U,V)\times \Pr(V;X)\lra \Pr(U;X)\end{equation}
given by sending a pair $(f,W)$ to $(f^{-1}(W),\alpha\circ f)$ is continuous.
\end{lemma}
\begin{proof}
Let $(f,(W,\alpha))\in \Emb(U,V)\times \Pr(V;X)$ and let $(K,\epsilon,A)$ be a neighbourhood of $(f^{-1}(W),\alpha\circ f)$. Let $K'\subset V$ be a compact subset containing a relatively compact open neighbourhood $B$ of $f(K)$. Let $(K,B)$ be the open neighbourhood of $f$ consisting of embeddings that send $K$ into $B$. Let $(K,\epsilon)$ be the open neighbourhood of $f$ consisting of embeddings for which 
\begin{align*}
\|f(x)-g(x)\|&<\epsilon &\text{ for all $x\in K$}\\
\|Df(x) - Dg(x)\| &< \epsilon &\text{ for all $x\in K$}.
\end{align*}
Then there are $\delta_0,\delta_1>0$ and an $A$ for which the image of $((K,B)\cap (K,\delta_0))\times (K',\delta_1,A)$ is contained in $(K,\epsilon,A)$. 
\end{proof}
These maps preserve the subspaces $\P(U)$ and $\Pr(U)$ (see Remark \ref{remark:100} \eqref{remit:10}). In particular, Lemma \ref{lemma:32} gives another proof of Lemma \ref{lemma:31} and of Section 2.2 in \cite{GR-W}.

\subsection{A more convenient assumption}
Let $B^n$ be the open unit ball in $\bR^n$. By the previous lemma, any diffeomorphism $f\colon B^n\to \bR^n$ %given by $f(x) = \frac{x}{1-\|x\|}$
induces a homeomorphism $\qq\to \pp$. Our theorem will be proven by showing that the inclusion $\widetilde{\mathcal{L}}_d(B^n)\to \qq$ is a homotopy equivalence, where $\widetilde{\mathcal{L}}_d(B^n)$ consists on intersections of affine planes with the unit ball. The reason for this change is that it will be convenient to make use of diffeomorphisms between $B^n$ and balls of different radius fixing the subspaces $\widetilde{\mathcal{L}}_d(B^n)$, and this is more complicated we take $\bR^n$ instead.

%Recall the concept of a Frechet mean from \cite{bla}\mnote{I haven't found any good reference for this}. %Let us denote by $\Power(X)$ the space of compact subsets of $X$ with the Hausdorff distance, and let $\Power^\eta(X)\subset \Power(X)$ be the subspace where the subsets have diameter at most $\eta$. \mnote{The following proposition probably exists in the literature??}
%For our purposes, its properties may be summarized in the following proposition: For a metric space $X$ and a number $\eta$, we declare $\Power^\eta(X)$ to be the space of compact subsets of diameter at most $\eta$ of $X$, with the natural topology coming from the Hausdorff distance.

\section{The space of almost linear submanifolds} \label{section:almost}
Recall that, for a submanifold $W$ of $B^n$, the Gauss map $$\Gauss\colon W\to \Gr_d(\bR^n)$$ sends a point to its tangent plane. Let us %fix once and for all a Riemannian metric on $\Gr_d(\bR^n)$ with diameter $1$, and
 write $\diam$ for diameter. 
\begin{df} The space of $\delta$-almost linear submanifolds of $B^n$ is the subspace $\pp^{\delta}\subset \qq$ of those submanifolds $W$ such that $\diam\circ \gauss (W)<\delta$. 
\end{df}
In this section we prove that the inclusion $\qq^{\delta}\subset \qq$ is a homotopy equivalence that fixes $\ll$ when $\delta>0$. 
\begin{remark}\label{remark:23} To prove that this inclusion is a weak homotopy equivalence is easier: Let us see first that it is surjective on components: if $W\in \qq$, then there exists an $\epsilon>0$ such that for all $0<a\leq \epsilon$, $(\frac{1}{a}\cdot W)\cap B^n$ is in $\qq^\delta$. Then, by \eqref{lemma:cont}, the isotopy $t\mapsto (x\mapsto \frac{1}{1+(\epsilon-1)t})$ induces a path from $(\frac{1}{\epsilon}\cdot W)\cap B^n$ to $W$. 

Now, given a lifting problem
\[\xymatrix{
\partial D^n \ar[d]\ar[r]^f & \qq^\delta \ar[d]\\
D^n\ar[r]^g & \qq},\]
let 
\[\epsilon(x) = \sup\{\epsilon\in (0,1]\mid \forall a<\epsilon, \frac{1}{a}f(x)\in \qq^\delta\}.\]
One can use that $D^n$ is compact to find an $\epsilon$ as above that works for all the values of $g$ at once. Then, using this $\epsilon$, one replaces this lifting problem by a homotopy equivalent lifting problem in which the lift exists on the nose.  
\end{remark}

\subsection{Another technical lemma} Let $F_{\geq}([0,1))$ be the set of bounded non-decreasing real functions from $[0,1)$ to $[0,\infty)$ that preserve $0$, and let $C_>([0,1))$ be the subset of continuous increasing functions. We topologize these spaces as follows: %with the almost everywhere convergence 
%(i.e. a neighbourhood basis of $f$ is given by pairs $(J,\epsilon)$, where $J\subset [0,1)$ is compact. A function $g$ belongs to such neighbourhood if there is a countable subset $A$ of $[0,1)$ and for all $x\ in K\setminus A$, it holds that $|f(x)-g(x)|<\epsilon$. 
%(i.e. $f_n\to f$ if there is a countable subset $A\subset [0,1)$ and for all $x\notin A$, $f_n(x)\to f(x)$). 
%We topologise $F_\geq([0,1))$ as the inverse limit of the restriction maps $F_\geq([0,t))\to F_\geq([0,s))$, where each space in the sequence is equipped with the graph topology. In other words, two functions are close if their graphs are close when restricted to a compact subset.
For each positive real number $\epsilon$, there is a basic neighbourhood $(\epsilon)^F$ of a function $f$. Another function $g$ is in the $(\epsilon)$-neighbourhood of $f$ if for all $s \in [\epsilon,1-\epsilon]$, $f(s-\epsilon)-\epsilon \leq g(s) \leq f(s+\epsilon)+\epsilon$. The subspace $C_>([0,1))$ has the compact-open topology.% and therefore its topology is finer than the $L_1$-topology. 

\begin{df} For each $f\in F_{\geq}([0,1))$, define $\rho(f)\colon [0,1)\to [0,\infty)$ as 
\[\rho(f)(x) = \int_x^{\sqrt{x}}{f(y)dy}.\]
\end{df}
\begin{lemma}\label{lemma:operator0} \begin{enumerate}
\item If $f\in F_{\geq}([0,1))$, then $\rho(f)\in C_{\geq}([0,1))$ and $f\leq \rho(f)$;
\item if $f\in F_>([0,1))$, then $\rho(f)\in C_>([0,1))$ and $f<\rho(f)$;
\item if $f=0$, then $\rho(f) = 0$.  
\end{enumerate}
\end{lemma}
\begin{proof}
The third point is obvious. For the other two, observe first that the value of $\rho(f)$ on $x$ is the mean value of $f$ in the inverval $[x,\sqrt{x}]$.
\begin{enumerate}
\item As $f$ is non-decreasing, the mean value of $f$ in $[x,\sqrt{x}]$ is non-decreasing too, and it is bigger or equal than the value of $f$ at the initial point of the interval.
\item If $f$ is strictly increasing, then the mean value of $f$ in $[x,\sqrt{x}]$ is strictly increasing too, and strictly bigger than the value of $f$ at the initial point of the interval.\qedhere
\end{enumerate}
\end{proof}

\begin{lemma}\label{lemma:operator} The operator $\rho\colon F_{\geq }([0,1))\to C_{\geq}([0,1))$ is continuous.
\end{lemma}
\begin{proof}
If $g$ is in the $\epsilon$-neighbourhood of $f$, then 
\begin{align*}
&\left|\int_{x}^{\sqrt{x}} f(y) dy - \int_x^{\sqrt{x}} g(y) dy\right| & \\
&\leq \int_{0}^{1} |f(y) - g(y)|dy &  \\ 
&\leq \int_{0}^\epsilon |f(y)-g(y)|dy + \int_{\epsilon}^{1-\epsilon} f(y+\epsilon)+\epsilon - (f(y-\epsilon)-\epsilon) dy + \int_{1-\epsilon}^{1} |f(y)-g(y)|dy &\\
& \leq \epsilon + \int_{\epsilon}^{1-\epsilon}2\epsilon dy + \int_{2\epsilon}^{1} f(y)dy - \int_{0}^{1-2\epsilon} f(y) dy + \epsilon &\\
& \leq 2\epsilon + 2\epsilon(1-2\epsilon) + \left(\int_{2\epsilon}^{1-\epsilon} f(y) dy +  \int_{1-\epsilon}^{1} f(y) \right)- \left(\int_{0}^{\epsilon} f(y) dy + \int_{\epsilon}^{1-2\epsilon} f(y) dy\right) + \epsilon &\\
%& \leq 2\epsilon + 2\epsilon(1-2\epsilon) + \left(\int_{2\epsilon}^{z-\epsilon} f(y) dy +  \epsilon \right)- \left(\int_{2\epsilon}^{z-\epsilon} f(y) dy \right)  &\\
& \leq \epsilon + 2\epsilon(1-2\epsilon) + \left(\int_{2\epsilon}^{1-\epsilon} f(y) dy -   \int_{\epsilon}^{1-2\epsilon} f(y) dy\right)+ \left(\int_{1-\epsilon}^{1} f(y) -\int_{0}^{\epsilon} f(y) dy \right)+ \epsilon&\\
&\leq 2\epsilon + 2\epsilon(1-2\epsilon) + (\epsilon) + (\epsilon) + \epsilon.&\qedhere 
\end{align*}
\end{proof}

If we let $\overline{X}$ denote the one-point compactification of a locally compact space $X$, then there is a continuous map 
\begin{align}\label{eq:oe}
\mathrm{O}\Emb(X,Y)\lra \map(\overline{Y},\overline{X})
\end{align}
from the space of open embeddings of $X$ into $Y$ to the mapping space between $\overline{Y}$ and $\overline{X}$, both endowed with the compact-open topology (this is an adaptation of the second part of the proof of Theorem 4 in \cite{Arens}, cf.\ \cite{Cantero:collapse-mathoverflow}). It is given by sending an embedding $e$ to the map that sends a point $y$ to $e^{-1}(y)$ if the latter exists and to $\infty$ otherwise.

The following lemma and its use was suggested to me by Abd\'o Roig.

%\begin{lemma} There is a continuous map $a\colon F_\geq(\bR^+) \lra C_>(\bR^+)$ such that $a(f)\geq f$, $a(f)(x)\geq \alpha x$ and $a(0)(x) = \alpha x$
%\end{lemma}
%\begin{proof}
%The function $b(f)(x) = \frac{1}{x}\int_x^{2x}{f(y)dy}$ is continuous and non-decreasing. In addition, being $f$ non-decreasing, $b(f)(x) \geq \frac{1}{x}(f(x)(2x-x)) = f(x)$. To make it strictly increasing, first define $g(x) = \alpha x$ and replace $f$ by $f + g$, which is always positive and bigger than $f$. Therefore $a(f) = b(f+g)$ gives the desired function.
%\end{proof}

\begin{lemma}\label{lemma:abdo} %If $\sigma\colon X\to F_{\geq}([0,1))$ is a continuous map and $\alpha>0$, then 
For each $\delta>0$, there exists a continuous function $\sigma\colon F_{\geq}([0,1)) \to (0,\frac{1}{2})$ such that $f(\sigma(f))<\delta$ for all $f$ and $\sigma(0) = 1$. %In addition $\sigma(p) = \frac{1}{2}$ if $\sigma(p)$ is the constant function with value $0$.
\end{lemma}
\begin{proof}
%If $f\in F_{\geq}([0,1))$, define the continuous function $\rho(f)\colon [0,\frac{1}{2}]\to [0,\infty)$:
%\[\rho(f)(x) = \frac{1}{x}\int_x^{2x}\left( f(y) + 2 y\right) dy.\]
%Now, $\rho(f)$ is continuous by Lemma \ref{lemma:operator}, and it is strictly increasing, because 1) the function $f(y) + 2y$ is positive when $y$ is positive, and 2) the value of $\rho(f)$ on $x$ is the mean value of $f(y) + \alpha y$ in the inverval $[x,2x]$. In addition, since $f(y)+\alpha y$ is non-decreasing, this mean value is bigger than the value of $f$ at $y$, so $\rho(f)\geq f$.
 %%It is strictly increasing because $f(y)+\alpha y$ is positive on $(0,1)$, and $\rho(f)(x) \geq \frac{1}{x}(f(x)+\alpha x)(2x-x) \geq f(x)$. Moreover, since $f(0)=0$ and it is continuous in a neighbourhood of $0$, it follows that $\rho(f)$ is also continuous in that neighbourhood and that $f(0) = 0$ by L'Hopital. 
%By Lemma \ref{lemma:operator}, this defines a continuous function $\rho\colon F_{\geq}([0,1))\lra C_{>}([0,\frac{1}{2}))$: 

It is easy to see that the assignment $f\mapsto (y\mapsto f(y)+\delta y)$ defines a continuous function
\[\tau\colon F_{\geq}([0,1))\longmapsto F_>([0,1))\]
with the property that $\tau(f)\geq f$. By Lemma \ref{lemma:operator}, composing it with $\rho$, we obtain a continuous function to $C_>([0,1))$.

Functions in $C_>([0,1))$ are not invertible, but they are open embeddings, hence the map \eqref{eq:oe} particularises in this case to a map
\[\mathrm{inv}\colon C_>([0,1))\overset{f\mapsto f^{-1}}\lra C_{\geq}([0,\infty]).\]
Therefore, if we write $\mathrm{ev}_\delta$ for the evaluation at $\delta$, then the composite
%\[\sigma\colon X\overset{\sigma}{\lra} 
\[\sigma\colon F_{\geq}([0,1))\overset{\tau}{\lra} F_>([0,1))\overset{\rho}{\lra} C_>([0,1)) \overset{\mathrm{inv}}{\lra} C_{\geq}([0,\infty]) \overset{\mathrm{ev}_\delta}{\lra} [0,1] \]
is continuous. Writing $g=\tau(f)$ and using that $f< g$ and that $g< \rho(g)$ (see Lemma \ref{lemma:operator0}),% if $f=\sigma(p)$ and $y = \sigma(p)$, then
\[f(\sigma(f)) = f(\rho(g)^{-1}(\delta))< g(\rho(g)^{-1}(\delta)) < \rho(g)(\rho(g)^{-1}(\delta)) = \delta.\]
In addition, we have that if $f=0$, then $\sigma(f) = \rho(\tau(0))^{-1}(\delta) = \rho(\delta x)^{-1}(\delta) = 1$.
%and if $f$ is constant with value $0$, then $\rho(f)(x) = \delta x$, and therefore
%\[y = \rho(f)^{-1}(\delta) = 1.\]
\end{proof}

\subsection{The space of $\delta$-linear submanifolds}
Let us define
\[\vartheta\colon \qq\lra F_\geq([0,1))\]
as the adjoint of the map $\qq\times [0,1)\to [0,\infty)$ that sends a pair $(W,\epsilon)$ to $\diam\circ\Gauss(W_\epsilon)$, where $W_\epsilon = W\cap B(\epsilon)$ and $B(\epsilon)$ is the ball of radius $\epsilon$ centered at the origin.%sends a submanifold $W$ and an $\epsilon>0$ to the diameter of the image of $W_\epsilon:= W\cap B(\epsilon)$ by the Gauss map, i.e. $\vartheta(W)(\epsilon) = \diam\circ \Gauss(W_\epsilon)$.

\begin{lemma} The map $\vartheta$ is continuous.
\end{lemma}
\begin{proof}
Let $(\epsilon)^F$ be a neighbourhood of $\vartheta(W)$. Let $K(r)\subset B^n$ be the closed disc with radius $r = 1-\epsilon$, and let $(K(r),\epsilon)$ be a neighbourhood of $W$. If $s<r$ and $W'\in (D(s),\epsilon)^{\widetilde{\Psi}}$, then, by condition (\ref{qui2}) and by \eqref{eq:5} in the definition of the topology of $\pp$ the followig holds: for each $y\in W'_{s}$, there is an $x\in W$ at distance at most $\epsilon$ (so $x\in W_{s+\epsilon}$), and such that $T_yW'$ and $T_xW$ are also at distance at most $\epsilon$. Therefore for all $s<r$: 
\[\diam\circ \Gauss(W'_s)\leq \diam\circ \Gauss(W_{s+\epsilon}) + \epsilon.\]
 Similarly, using \eqref{eq:5},
\[\diam\circ \Gauss(W'_s)\geq \diam\circ\Gauss(W_{s-\epsilon}) - \epsilon.\]
 As a consequence, if we write $f=\vartheta(W)$ and $g=\vartheta(W')$,
\[f(s-\epsilon)-\epsilon  \leq g(s)\leq  f(s+\epsilon) + \epsilon.\]
Hence $\vartheta((D(r),\epsilon)^\ls)\subset (\epsilon)^F$.
\end{proof}

%\begin{df} Let $\varphi\colon \qq\to (0,\infty)$ the continuous function obtained applying $\vartheta$, then $a$, then taking the inverse of the resulting function, and finally evaluating that function on $\alpha$. It has the following properties: 
%\begin{align}
%\diam\circ\Gauss(W_{\varphi(W)})&<\alpha & \text{for all $W$} \label{diam-gauss-varphi}\\
%\varphi(W)&=1 &\text{if $W$ is a union of parallel planes}.\label{varphi-linear} \\
%\varphi(W)&<1 &\text{if $W$ is not a union of parallel planes}.\label{varphi-nonlinear}
%%W_{\varphi(W)}&\in \Pr(\frac{1}{\varphi(W)}B^n)^\delta \label{compactifiable}
%\end{align}
%\end{df}

\begin{proposition}\label{prop:2} For each $\delta>0$, the inclusion $\qq^\delta\subset \qq$ has a homotopy inverse through a homotopy that fixes $\qq^0$ pointwise. %$H_t\colon \qq\to \qq$ that fixes $\qq^0$ which is the identity at time $1$ and has image in $\qq^\delta$ at time $0$.
\end{proposition}
\begin{proof}
For a positive real number $r$, let $f_{r,t}\in \Emb(B^n,B^n)$ be the isotopy of embeddings $f_{r,t}(x) = (t+(1-t)r)x$, and let $f_r = f_{r,0}$. Let $\sigma$ be given by Lemma \ref{lemma:abdo}. Define $h\colon \qq\to \qq^\delta$ by $W\mapsto f_{\sigma\circ\vartheta(W)}^{-1}(W)$, which is continuous by Lemma \ref{lemma:cont}, and is well-defined because
\[h(W) = \frac{1}{\sigma\circ \vartheta(W)}(W\cap B(\sigma\circ\vartheta(W)))\]
and
\begin{align*}\diam\circ\gauss(\frac{1}{\sigma\circ\vartheta(W)}(W\cap B(\sigma\circ\vartheta(W))))& = \diam\circ\gauss(W\cap B(\sigma\circ \vartheta(W))) \\
& = (\vartheta(W))(\sigma\circ \vartheta(W)) \leq \delta\end{align*}
(the latter inequality follows from Lemma \ref{lemma:abdo}). 

By Lemma \ref{lemma:cont}, the map $H_t\colon \qq\to \qq$ (resp.\ $G_t\colon \qq^\delta\to\qq^\delta$) that sends a submanifold $W$ to $f_{\sigma\circ\vartheta(W),t}^{-1}(W)$ defines a homotopy between $i\circ h$ and the identity (resp.\ $h\circ i$ and the identity). It is clear that $\qq^0$ remains fixed because the value of $\varpi$ on an element $W$ of $\qq^0$ is the constant function with value $0$, and so the value of $\sigma$ on $\varpi(W)$ is $1$.
%
%
%
%Now, $f_{r,t}$ is conformal, and this implies that $\Gauss(W_{\bar{\vartheta}(W)}) = \Gauss(H_0(W))$, and it follows from \eqref{diam-gauss-varphi} that $H_0(W)$ satisfies \eqref{eq:101}. In addition, $H_0(W)$ is compactifiable as a consequence of \eqref{varphi-nonlinear} (observe that unions of affine planes are always compactifiable). Hence $H_0(W)\in \qq^\delta$.
\end{proof}

\section{The space of linear submanifolds} %In this section we prove that the inclusion $\widetilde{{\mathcal L}}_d(B^n)\subset \qq^{\delta}$ is a deformation retract. %(i.e.\ we extend Proposition \ref{prop:2} to the case when $\delta = 0$).

Let $\bump\colon [0,1)\to [0,1]$ be a bump function that takes value $1$ in a neighbourhood of $0$ and value $0$ on $[\frac{1}{2},1]$ and whose slope is bounded by $7$. If $x\in B^n$, we write $\bump(x):= \bump(\|x\|)$.

%\subsection{Averaging the tangent planes of an almost linear submanifold} 
\begin{df} For each non-empty $W\in \qq$ and each plane $L\in \Gr_d(\bR^n)$, define 
\begin{equation}\label{eq:929}\lambda(W,L) = \frac{1}{2\int_{W} \bump(x) dx}\int_{W} \bump(x)d(L,T_xW)^2dx.\end{equation}
\end{df}
As a consequence of Karcher's theorem \cite[Theorem 1.2]{Karcher}, we have:
\begin{proposition}\label{prop:karcher} There exists a $\delta>0$ (that depends only on the metric of $\Gr_d(\bR^n)$) such that if $\diam\circ\Gauss(W)<\delta$, 
%\begin{equation}\label{eq:101}
%\exists y\in \Gr_d(\bR^n)\text{ with } \Gauss(W)\subset B_\delta(y)%\quad \text{(for some ball $B_\delta$ of radius $\delta$),}
%\end{equation}
then for all $L\in \Gauss(W)$, the function $\lambda(W,-)$ is convex in $B_\delta(L)$. 
\end{proposition}
Observe in addition that the global minima of $\lambda(W,-)$ lie in the convex hule of $\Gauss(W)$, which is contained in $B_\delta(L)$, and therefore $\lambda(W,-)$ has a unique global minimum.

\begin{df} Let $\delta$ be given by Proposition \ref{prop:karcher} theorem, and define the map
\[\mu\colon \qq^\delta\setminus \{\emptyset\}\lra \Gr_d(\bR^n)\]
whose value on a submanifold $W$ is the minimum of $\lambda(W,-)$ (cf.\ with the concept of \emph{selection} \cite[p.\ 154]{Michael}).
\end{df}
The main property of this function is that it assigns to each submanifold $W$ whose Gauss map has small diameter, a plane $\mu(W)$ that is near each tangent plane of $W$ (so its orthogonal complement $\mu(W)^\perp$ is transverse to $W$):
\begin{align}\label{d-mu}
d(\mu(W),T_xW) &\leq \diam\circ\Gauss(W) \text{ for all $x\in W$.}
\end{align} 
because otherwise, evaluation on \eqref{eq:929} gives that
\[\lambda(W,T_xW) < \lambda(W,\mu(W)) \text{ for some $x\in W$},\]
contradicting the definition of $\mu(W)$.
\begin{proposition}\label{prop:4} The map $\mu$ is continuous.
%The assigment $\mu$ defines a continuous map from the subspace $\qq^\delta\subset \qq$ of compactifiable submanifolds satisfying \eqref{eq:101} to the space $\Gr_d(\bR^n)$.  
\end{proposition}
\begin{proof}
We will prove that the function $\lambda(-,-)$ is continuous. Then, because $\Gr_d(\bR^n)$ is a compact metric space, the adjoint $\alpha\colon \qq\to \map(\Gr_d(\bR^n),\bR)$ is continuous too. The target of this map is actually the subspace $X$ of those maps that
\begin{enumerate}
\item have a unique global minimum,
\item are convex in the ball of radius $\delta$ around the global minimum.
\end{enumerate}
Taking the global minimum defines a continuous function $\beta\colon X\to \Gr_d(\bR^n)$, and therefore $\mu = \beta\circ \alpha$ is continuous as well.

Let $(W,L)\in \Pr(U)\times \Gr_d(\bR^n)$, and let $(K_t,\epsilon)^\ls$ be a neighbourhood of $W$, where $K$ is a disc of radius $t$. Let $W'\in (K,\epsilon)^\ls$ and let $L'\in B_\eta(L)$, the ball of radius $\eta$ centered at $L$. Then, by Definition \ref{df:ls}, there is a covering map $q\colon Q\subset NW\to W$ over $W\cap K$ such that
\begin{enumerate}
\item $\exp_W(Q)\cap K \subset W'\cap K$,
\item $\|f(x)\| + \|\tau\circ(Df)(x)\| < \epsilon$ for each local section $f$ of the covering map.
\end{enumerate}
Let $t\geq \frac{3}{4}$ and let $\epsilon < \frac{1}{4}$ and define $V:= W\cap K_t$ and $V' := q^{-1}(W_{t})$. Then:
\begin{align*}
\int_{W} \bump(x)d(L,T_xW)^2dx &= \int_{V} \bump(x)d(L,T_xW)^2dx \\
\int_{W'} \bumpp(x)d(L,T_xW)^2dx &= \int_{V'} \bump(x)d(L,T_xW)^2dx
\end{align*}
because $\bump(x) = 0$ if $\|x\|>\frac{1}{2}$ and $W\cap K_{\frac{1}{2}}\subset V\subset W$ and $W'\cap K_{\frac{1}{2}}\subset V'\subset W'$. 
Therefore we have that the difference $\left|\lambda(W,L)-\lambda(W,L)\right|$ is
{\footnotesize
\begin{align*}
=&\left|\frac{1}{\int_W \bump(x) dx}\int_{W} \bump(x)d(L,T_xW)^2dx - \frac{1}{\int_{W'} \bumpp(y) dy}\int_{W'} \bumpp(y)d(L',T_yW')^2dy\right|&\\
=&\left|\frac{1}{\int_V \bump(x) dx}\int_{V} \bump(x)d(L,T_xW)^2dx - \frac{1}{\int_{V'} \bumpp(y) dy}\int_{V'} \bumpp(y)d(L',T_yW')^2dy\right|&\\
=&\left|\frac{1}{\int_V \bump(x) dx}\int_{V} \bump(x)d(L,T_xW)^2dx - \frac{1}{\int_{V'} \bumpp(y) dy}\int_{V} \sum_{y\in q^{-1}(x)}\bumpp(y)d(L',T_yW')^2\|\det J_y(q)\|dx\right|.&
\end{align*}}
Now, we write $A = \int_V \bump(x) dx, A' = \int_{V'} \bumpp(x) dx$:
\begin{align*}
&=\frac{1}{AA'}\left|\int_{V} A'\bump(x)d(L,T_xW)^2dx - \int_V\sum_{y\in q^{-1}(x)} A\bumpp(y)d(L',T_yW')^2\|\det J_y(q)\| dx\right|&
\end{align*}
Now, let $k$ be the cardinality of $q^{-1}(x)$, and replace the second ocurrence of $A'$ by $kA + |A'-kA|$:
{\footnotesize \begin{align*}
&\leq\frac{1}{AA'}\left|\int_{V} (kA+|A'-kA|)\bump(x)d(L,T_xW)^2dx - \int_V \sum_{y\in q^{-1}(x)}  A\bumpp(y)d(L',T_yW')^2\|\det J_y(q)\| dx.\right|&
\end{align*}}
Now, move the term $kA\bump(x)d(L,T_xW)^2$ to the second integral:
{\footnotesize \begin{align*}
&\leq\frac{1}{AA'}\left|\int_{V} |A'-kA|\bump(x)d(L,T_xW)^2dx - \int_{V} \sum_{y\in q^{-1}(x)} |A\bump(x)d(L,T_xW)^2 - A\bumpp(y)d(L',T_yW')^2\|\det J_y(q)\| |dx\right|&
\end{align*}}
rearranging the $A$ and the $A'$:
{\footnotesize
\begin{align*}
&=\frac{|A'-kA|}{A'}\int_V \bump(x)d(L,T_xW)^2dx - \frac{1}{A'}\int_{V} \sum_{y\in q^{-1}(x)} |\bump(x)d(L,T_xW)^2 - \bumpp(y)d(L',T_yW')^2\|\det J_y(q)\| |dx&
\end{align*}}
and using that $ab-cd = a(b-d) + (a-c)d$ with
\begin{enumerate}
\item $a = \bump(x)$,
\item $b = d(L,T_xW)^2$
\item $c = \bumpp(y)\|\det J_y(q)\|$
\item $d = d(L',T_yW')^2$
\end{enumerate} we have that the last display is
{\footnotesize
\begin{equation}\label{eq:p0}
\begin{split}
&\frac{|A'-kA|}{A'}\int_V \bump(x)d(L,T_xW)^2dx\\ 
&-\frac{1}{A'}\int_{V} \sum_{y\in q^{-1}(x)} \left|\bump(x)|d(L,T_xW)^2 - d(L',T_yW)| + |\bump(x) - \bump(y)\|\det(J_y(q))\||d(L',T_yW')^2 \right|dx
\end{split}
\end{equation}}
First, using the triangle inequality one has that
\[|d(L,T_xW) - d(L',T_yW')|\leq d(L,L')+d(T_xW,T_yW'),\]
and both 
If we write this inequality as $|a-b|\leq c+d$, multiplying both sides by then we have that
\[|a^2-b^2|\leq (a+b)(c+d)\]
and here $c<\eta$, $d<\tan(\epsilon)<\epsilon$ (see Remark \ref{remark:100} \eqref{item:100}) and $a,b\leq \pi/2$. In addition, $\bump(x)\leq 1$ therefore
\begin{equation}\label{eq:p1}\bump(x)|d(L,T_xW) - d(L',T_yW')|\leq \pi\cdot (\epsilon+\eta).\end{equation}
%3\epsilon^2 + 4\epsilon
Second, using that the slope of $\bump$ is at most $7$, that $|\|x\|-\|y\||<\epsilon$ and that $\nu(x)\leq 1$, we have that
\begin{align*}
|\bump(x) - \bump(y)\det(J_y(q))| &\leq |\bump(x) - \bump(y)| + |\bump(y)(1-\det J_y(q))| \\
&\leq 7\epsilon + |1-\det J_y(q))\end{align*}
%and because $\bump(y)\leq 1$ and $|1-\det J_y(q)|\leq \epsilon$ and $d(L',T_yW')^2\leq 1$, 
%\begin{equation}\label{eq:p2}|\bump(x) - \bump(y)\det(J_y(q))|d(L',T_yW')^2\leq 14\epsilon + \epsilon'.\end{equation}
As a consequence (and using that $d(L,T_xW)^2\leq \pi/2$), \eqref{eq:p0} is bounded by
\begin{equation}\label{eq:31}\frac{|A'-kA|}{A'}A + k\pi\frac{((\epsilon+\eta)+(7\epsilon+|1-\det J_y(q)|)/2}{A'}{\mathrm{vol}}(V).\end{equation}
\def\vol{\mathrm{vol}}
Finally, $|A'-kA|$ is bounded by
\[\int_{V'} \bump(x) dx - k\int_V \bump(x)dx = \int_{V} \sum_{y\in q^{-1}(x)} \bump(y) dx - k\int_V \bump(x)dx = \int_V\sum_{y\in q^{-1}(x)}(\bump(x)-\bump(y)) dx.\]
Since $\bump$ has slope at most $7$ and $d(x,y)<\epsilon$, we have that the latter expression is bounded by $7k\epsilon\mathrm{vol}(V)$. Using this, it follows that $A'\geq |kA - 7k\epsilon \vol(V)| \sim kA$ if $\epsilon$ is very small. Hence $\eqref{eq:31}$ is bounded by:

%$|1-\frac{A'}{kA}| = \frac{|kA-A'|}{kA}\leq \frac{7\epsilon\vol(V)}{A}$, and therefore 
\begin{align*}
\frac{7k\epsilon \vol(V)}{|kA-7k\epsilon\vol(V)|}A + k\pi\frac{((\epsilon+\eta)+(7\epsilon+|1-\det J_y(q)|)/2}{A'}{\mathrm{vol}}(V)
\end{align*}
which can be made arbitrarily small taking $\epsilon$ and $\eta$ small enough (notice that $|1-\det J_y(q)|$ is bounded by a continuous function of $\epsilon$ which does not depend on $W'$).
%$\epsilon$ and $d(L,L')$ are small enough, we can assume that:
%\begin{enumerate}
%\item $d(L,T_xW)^2 \leq d(L,L')^2 + d(L',T_yW')^2 + d(T_yW',T_xW)^2\leq d(L,L')^2 + d(L',T_yW')^2 + d(T_yW',T_xW)^2$
%
%
%\item $|1-|\det J_y(q)||<\delta$,
%\item therefore $|A'-kA|<\delta A$,
%\item by property (2) in the definition, $|d(L,T_xW)^2 - d(L',T_yW')^2| \leq |d(L,L')^2 + d(L',T_yW')^2 + d(T_yW',T_xW)^2 - d(L',T_yW')^2| \leq |d(L,L') + d(T_yW',T_xW)^2|< 2\delta$
%\item by property (2) in the definition $|\bump(x)| - |\bumpp(y)\det J_y(q)|\leq |\bump(x)| - |\bumpp(y)| + |\delta\bumpp(y)| \leq \delta + \delta$, if $q(y) = x$.
%\end{enumerate}
%Therefore we get that the latter term is bounded ...
\end{proof}
%if $\nabla_W:= \int_W(1-\|x\|)dx$, then
%\begin{align*}
%\frac{1}{2\nabla_{W'}}\int_{W'} (1-\|x\|)d(L,T_xW')^2dx =\\
 %\frac{1}{2\nabla_{W'}}\left(\int_{W'_t} (1-\|x\|)d(L,T_xW')^2dx + \int_{W'\setminus W'_t} (1-\|x\|)d(L,T_xW')^2dx \right) \\
%\frac{1}{2\nabla_{W'}}\left(\int_{W'_t} (1-\|x\|)(d(L,T_{q(x)}W) + \alpha(x))^2dx + \beta(x) \right)\\
%\frac{\#\text{sheets of $q$}}{2\nabla_{W'}}\left(\int_{W_t} (1-\|q(x)\|+\nu(x))(d(L,T_{x}W) + \alpha(x))^2\gamma(\epsilon)dx + \beta(x) \right)
%\end{align*}
%and $|\alpha(x)|<d(T_xW',T_{q(x)}W)<\epsilon$ and $0\leq\beta(x)\leq \nabla_{W'\setminus W'_t}(1-t)\delta^2$, $\gamma(\epsilon) = \det(Dq)$ and $|\nu(x)|<\epsilon$. Therefore, when $t\to 1$ and $\epsilon\to 0$, the above integral converges to
%\begin{align*}
%\frac{\#\mathrm{sheets}}{2\nabla_{W'}}\int_{W} (1-\|x\|)d(L,T_{x}W)^2dx
%\end{align*}
%and similarly $\nabla_{W'}$ tends to $\#\mathrm{sheets}\cdot \nabla_W$ as $t\to 1$ and $\epsilon\to 0$.
%\end{proof}

%\begin{df}Let $\alpha<\pi/4$ be smaller than the $\delta$ provided by Karcher's theorem, and also smaller than half the distance between any two planes $P,Q\in \Gr_d(\bR^n)$ such that $P\cap Q^\perp\neq 0$. 
%\end{df}

\begin{proposition}\label{prop:3} Let $\delta$ be given by Proposition \ref{prop:karcher}, and assume that $\delta<\pi/2$. Then the inclusion $\widetilde{{\mathcal L}}_d(B^n)\subset \qq^{\delta}$ is a deformation retract.
\end{proposition}
\begin{proof}
If $P\in \Gr_d(\bR^n)$, and $\pi,\pi^\perp$ are the projections onto $P$ and its orthogonal complement, we let $g_{t,P}(x) = t\cdot \pi(x) + \pi^\perp(x)$ for all $x\in B^n$. This defines a continuous map
\[g\colon [0,1]\times \Gr_d(\bR^n)\times B^n \lra B^n.\]
%then define an isotopy of embeddings $g_{t,P}\colon B^n\to B^n$ by $g_{t,P}(x) = t\cdot \pi(x) + \pi^\perp(x)$. 
Let $G_t\colon \qq^\delta \to \qq$ be the homotopy 
\[
G_t(W) = \begin{cases}
g_{t,\mu(W)}^{-1}(W) & \text{ if $t>0$} \\
\bigcup_{x\in \mu(W)^\perp\cap W} (x+\mu(W))\cap B^n & \text{ if $t=0$} \\
\emptyset & \text{ if $W=\emptyset$},
\end{cases}
\]
where $\mu(W)^\perp\cap W$ is discrete because $W$ intersects $\mu(W)^\perp$ transversely:  By \eqref{d-mu}, we have that for all $x\in W$,
\begin{equation}\label{eq:09}
d(T_xW,\mu(W))\leq \diam\circ\Gauss(W)<\delta<\pi/2.
\end{equation}
By Lemma \ref{lemma:31} and because $g$ and $\mu$ are continuous, $G_t$ is continuous for $t>0$. Let us see now what happens when $t=0$.

%The submanifold $G_0(W)$ is well-defined for $t=0$ if and only if $W$ intersects $\mu(W)^\perp$ transversally, in which case $g_{0,\mu(W)}^{-1}(W)$ is a union of affine planes parallel to $P$ whose origins (their closest points to the origin of $B^n$) are precisely $\mu(W)^\perp\cap W$. This is actually the case:

%Let $(K,\epsilon)^\ls$ be a neighbourhood of $G_0(W)$.
 
%Let $s>0$ be such that $V:= B_s(\mu(W)^\perp)\cap W$ is a union of balls. Since $T_xW\perp \mu(W)^\perp$ for all $x$, it follows that the projection $q\colon \exp^{-1}(V)\subset NG_0(W)\to G_0(W)$ is a covering map onto its image, that we denote $A$. Define now $V_t := G_t(V) = B_{\frac{s}{t}}(\mu(W)^\perp)\cap G_t(W)$, so that the projection $q_t\colon \exp^{-1}(V_t)\subset NG_0(W)\to G_0(W)$ is a covering map onto its image, denoted $A_t$. 

%We have that $NG_0(W) = G_0(W)\times \mu(W)^\perp$, so a section of the normal bundle is just a map $G_0(W)\to \mu(W)^\perp$. If $f$ is a local section of $q$, then $f_t(x) = f(t\cdot x)$ is a local section of $q_t$ with domain $A_t$ and it holds $\|f_t(x)\| = \|f(tx)\|$ and that $\|\tau\circ (Df_t)(x)\| = t\cdot \|\tau\circ (Df)(x)\|$.

We identify $NG_0(W) = G_0(W)\times \mu(W)^\perp\subset \bR^n$, so that the exponential map $\exp_{G_0}^{-1}\colon B^n\to \bR^n$ is the inclusion. Then by \eqref{eq:09}, we have that the restriction $q_t$ of the projection of $NG_0(W)$ onto $G_0(W)$ to $G_t(W)$ is a covering map and that if $f_t$ is a local section of $q_t$, 
\begin{align}\label{eq:449}
\begin{split}
\|f_t(x)\| &= f_1(tx) = d_0(tx,G_0(W))\\
\|\tau\circ (Df_t)(x)\|  &= t\cdot \|\tau\circ (Df_1)(tx) \|  \\
&= t\cdot \tan(d_1(\mu(W),T_{tx}(W))) \leq  t\cdot\tan(\delta).
\end{split}
\end{align}
For each $\epsilon>0$, let $r_\epsilon$ be so small that if $x\in B_{r_\epsilon}(\mu(W)^\perp)\cap W$, then $d_0(x,G_0(W))<\epsilon/2$. Then by the formula above, if $t<r_\epsilon$, then $\|f_t(x)\|<\epsilon/2$ for all $x$. If in addition $t<\frac{\epsilon}{2\tan(\delta)}$, then we have that $\|\tau\circ (Df_t)(x)\|<\epsilon/2$. Therefore, if $(K,\epsilon)$ is a neighbourhood of $G_0(W)$,
\[t<\lambda_\epsilon:= \min\{r_\epsilon,\frac{\epsilon}{2\tan(\delta)}\}\Rightarrow G_t(W)\in (K,\epsilon).\]

 %which contains $A:= B_{r_\epsilon}(\mu(W)^\perp)\cap G_0(W)$. We write $q_t$
%
 %and if $f$ is a local section of a point in $A$, then $f_t(x) = f(t\cdot x)$ is a 
%
%
%
%\begin{align*}
%\|f(x)\| = d_0(x,G_0(W))<\epsilon/2
%\end{align*}
%
%
%Let $\lambda_\epsilon:= \min\{r_\epsilon,\frac{\epsilon}{2\tan(\delta)}\}$. If $t<\lambda_\epsilon$, then $A_t = G_0(W)$ and using Remark \ref{remark:100} \eqref{item:100}, we have that if $x\in G_0(W)$ and $y=\exp_{G_0(W)}f_t(x)$,
%\begin{align}\label{eq:449}
%\begin{split}
%\|f_t(x)\| &\leq \epsilon/2 \\
%\|\tau\circ (Df_t)(x)\|  &= t\cdot \|\tau\circ (Df)(x)\| 
%\leq  t\cdot\tan(\delta) < \epsilon/2.
%\end{split}
%\end{align}
%%i.e., if $t<\lambda_\epsilon$ is as above, then $d(G_t(W),G_0(W))<\epsilon$.
%%\begin{align}\label{eq:449}
%%\begin{split}
%%d_0(x,y) &= \|f_t(x)\| \leq \epsilon/2 \\
%%d_1(T_xG_0(W),T_yV_t) &= \arctan\left(\|\tau\circ (Df_t)(x)\|\right)  = \arctan(t\cdot \|\tau\circ (Df)(x)\|) \\
%%&\leq \arctan( t\cdot\tan(\delta)) < \arctan(\epsilon/2)<\epsilon/2.\end{split}\end{align*}
If $(K,\epsilon)^\ls$ is a basic neighbourhood of $W$, we will denote it by $(K,\epsilon)_W$ during the rest of the proof. Observe that if $\epsilon'>0$ and $K$ is compact, then 
\[G_t((K',\epsilon')_{W}) \subset (K',\epsilon')_{G_t(W)},\]
\[(K',\epsilon')_{G_t(W)}\subset (K,\epsilon)_{G_0(W)} \Rightarrow (K',\epsilon')_{G_s(W)}\subset (K,\epsilon)_{G_0(W)} \text{ for all }s\leq t.
\]
Let $(K,\epsilon)_{G_0(W)}$ be a neighbourhood of $G_0(W)$. We have seen that $G_{t}(W)\in (K,\epsilon)_{G_0(W)}$ for all $t\leq \lambda_\epsilon$, so we may pick a neighbourhood $(K',\epsilon')_{G_{\lambda_\epsilon}(W)}$ of $G_{\lambda_\epsilon}(W)$ contained in $(K,\epsilon)_{G_0(W)}$. Then, for all $t<\lambda_\epsilon$, we have that $(K',\epsilon')_{G_{t}(W)}$ is a neighbourhood of $G_t(W)$ that is contained in $(K,\epsilon)_{G_0(W)}$. 

Then we have that $G(t,W)=G_t(W)$ sends the neighbourhood $[0,\lambda_{\epsilon}/2)\times (K',\epsilon')_{W}$ of $(0,W)$ into the neighbourhood $(K,\epsilon)_{G_0(W)}$ of $G_0(W)$. Hence $G_t$ is continuous around $t=0$.

In particular, $g:= G_0$ is well-defined and lands in $\widetilde{{\mathcal L}}_d(B^n)$, and $i\circ g$ is homotopic to the identity through the homotopy $G_t$. On the other hand, $g\circ i$ is the identity because the value of $\mu$ on an affine plane $P$ is obtained by translating $P$ to the origin, and so $G_t$ applied to a plane is constant in $t$, hence $g$ restricts to the identity on $\widetilde{{\mathcal L}}_d(B^n)$.
%see \eqref{varphi-linear}), 
\end{proof}
\section{The homotopy type of the space of linear submanifolds}
\begin{proposition}\label{prop:5} (cf. \cite[Lemma~6.1]{galatius-2006}) The space $\widetilde{{\mathcal L}}_d(B^n)$ is weakly contractible.
\end{proposition}
\begin{proof}
Let $C_d(\bR^n)\subset \ll$ be the subspace of those non-empty unions of affine planes, all of whose origins (i.e.\ their closest points to the origin of $\bR^n$) are at distance at most $1$ from the origin of $\bR^n$. This is a closed subset. Let $U\subset \ll$ be the subspace of those (possibly empty) unions of planes that do not contain the origin. Then, there is a pushout square
\[\xymatrix{
U\cap C_d(\bR^n) \ar[d]\ar[rr] && U\ar[d] \\
C_d(\bR^n) \ar[rr] && \ll
}\]
which is also a homotopy pushout square because the upper horizontal arrow is a cofibration. Now, a point in $C_d(\bR^n)$ is a collection of parallel planes, and remembering the underlying linear plane defines a map $C_d(\bR^n)\to \Gr_d(\bR^n)$ that is also a fibre bundle. Its fibre over a plane $P$ is the space $C_0(P^\perp)$, which is the Ran space of $P^\perp$ and is well-known to be weakly contractible %\cite{BD}, 
\cite[Appendix]{Gaitsgory}. Therefore $C_d(\bR^n)\simeq \Gr_d(\bR^n)$. The same argument proves that $U\cap C_d(\bR^n)\simeq \Gr_d(\bR^n)$, and since the left vertical map is a map over $\Gr_d(\bR^n)$, it is a weak homotopy equivalence. On the other hand, $U$ is contractible as the homotopy $(W,t)\mapsto \frac{1}{1-t}W$ defines a contraction of $U$ to the empty submanifold. As a consequence, the homotopy pushout $\ll$ is weakly contractible as well.%\hfill $\square$
\end{proof}
%\section{Three more spaces}
%This section may be removed unless there is something interesting to say. 
%\subsection{Submanifolds of $\bR^n$ with the $\widetilde{\Psi}$-topology, together with a (locally constant) function to a discrete abelian monoid, so that when points collide, their labels are added}. The same argument as in the previous section gives the analogous theorem. 
%
%\subsection{Immersed submanifolds of $\bR^n$} i.e. images of immersions. This space is a commutative topological monoid with the property that $x^2=x$ for every point, and therefore it is weakly contractible by the same argument as for the Ran space.
%
%\subsection{Immersed submanifolds of $\bR^n$ with labels in an abelian monoid}. I can't figure this out.

\begin{proposition}\label{prop:6} The subspace $\Th(\gamma_d^\perp(B^n))$ of $\cL_d(B^n)$ that consists of connected or empty submanifolds, is a strong deformation retract. 
\end{proposition}
\begin{proof}
If $W\in \cL_d(B^n)$, let $W_{\mathrm{first}}$ be the closest component of $W$ to the origin of $B^n$, and let $f(W)>0$ be the distance from $W\setminus W_{\mathrm{first}}$ to the origin. This is a continuous function $f\colon \cL_d(B^n)\to [0,1]$. Then $h(W) = \frac{1}{f(W)}W$ defines a map $h\colon \cL_d(B^n)\to \Th(\gamma_d^\perp(B^n))$ and $h\circ i$ is the identity and $i\circ h$ is homotopic to the identity through the homotopy $H_t(W) = \frac{1}{(1-t)+tf(W)}W$.
\end{proof}

\section{Microflexibility}\label{s:4}\mnote{fc:This section is complete but needs a revision}

The spaces $\P(U)$ and $\Pr(U)$ glue together to form sheaves $\P(-)$ and $\Pr(-)$ on $\bR^n$: If $U\subset U'$ is a pair of open subsets then the restriction map $\Pr(U')\to \Pr(U)$ sends a submanifold $W$ to the intersection $W\subset W'$, and we have proven that these are continuous in families in Lemmas \ref{lemma:31} and \ref{lemma:32}. A sheaf of topological spaces in $\bR^n$ extends canonically to a sheaf of topological spaces in the site of manifolds and open embeddings \cite[Theorem~3.3]{R-WEmbedded}. 

%Theorem 3.3 (and Lemmas 3.4,3.5,3.6) explains how to promote $\Psi_d$ to a sheaf on the site of manifolds and open embeddings.
%
%Theorem 4.2 explains how to relate it to the cobordism category (this is the content of section 4)
%
%Section 5 explains why it is microflexible.
%
%Section 6 is a corollary of section 5, and the corollary that we want here.

At this point, one is tempted to use the methods in the latter article to extend our theorem to the space of merging submanifolds in an arbitrary open manifold: In that article, it was proven that the sheaf $\Psi(-)$ on a manifold $M$ is $\Diff(M)$-equivariant and that it is \emph{microflexible}. By a theorem of Gromov \cite{Gromov}, this automatically implies that for connected non-compact manifolds $M$ a certain map
\[\Psi(M)\lra \Gamma(\Psi^\fib(TM)\to M)\]
is a homotopy equivalence. The space on the right is the space of sections of the fibrewise space of submanifolds of the tangent bundle of $M$. By the Galatius--Randal-Williams theorem, the fibre over each point is homotopy equivalent to the Thom space $\Th\gamma_{d,n}^\perp$. 

The sheaves $\Pr(-)$ and $\Pg(-;X)$ are in fact $\Diff(M)$-equivariant, and if they were also microflexible, then one could deduce that certain maps
\begin{align*}
\Pr(M)&\lra \Gamma(\Pr^\fib(TM)\to M) \\
\Pg(M)&\lra \Gamma(\Pg^\fib(TM;X)\to M)
\end{align*}
are homotopy equivalences. But this is a castle in the sky:

\begin{proposition} The sheaves $\Pr(-)$ and $\Pg(-;X)$ are not microflexible.
\end{proposition}
Let us recall first the definition of microflexibility.
\begin{df} A sheaf $\Phi$ on a manifold $M$ is microflexible if for each pair $C'\subset C$ of compact subspaces of $M$, and each pair $C'\subset U',C\subset U$ of open subsets of $M$ such that $U'\subset U$, and for each diagram
\[\xymatrix{
P\times \{0\}\ar[r]^f\ar[d] & \Phi(U)\ar[d]^r \\
P\times [0,1]\ar[r]^h & \Phi(U')}\]
there exists an $\epsilon>0$ and a pair of open subsets $C'\subset V'\subset U'$ and $C\subset V\subset U$ such that $V'\subset V$, and a dashed arrow
\[\xymatrix{
P\times \{0\}\ar[r] \ar[d]& \Phi(U)\ar[r] & \Phi(V) \ar[d]  \\
P\times [0,\epsilon)\ar[r]\ar@{-->}[urr] & \Phi(U')\ar[r] & \Phi(V')
}\]
\end{df}
\begin{proof}
As for the main theorem, we only give the proof for the sheaf $\Pr(-)$, the proof for $\Pg(-;X)$ being exactly the same. 

Let $W'$ be a connected compact submanifold of $\bR^n$, and let $C'=U'$ be a tubular neighbourhood of $W'$ (which we implicitely identify with $NW$ from now on).

Let $W''\subset W\subset W'$ be codimension $0$ submanifolds such that $W''$ is closed as a subset and $W$ is open, and the first inclusion is a homotopy equivalences.

Let $U$ and $C$ be the restrictions of $U'$ to $W$ and $W''$ respectively. 

Let $C_k(NW)$ be the fibrewise configuration space of $k$ unordered points in the normal bundle of $W$. A section $f$ of this bundle defines 
\begin{enumerate}
\item a $k$-sheeted covering of $W$ and 
\item an element in $\Pr(NW)\subset \Pr(U)$.
\end{enumerate}
Since the fibre of $NW$ is a vector space, we can multiply any subset of it by a real number. Then we can define a path
\[[0,1]\lra \Pr(NW)\]
by sending $t>0$ to $t\cdot f(W)$ and $t=0$ to $W$. If we have a family of sections of $C_k(NW)$ indexed by $P$, we obtain a map
\[g\colon P\times [0,1]\lra \Pr(NW)\]
whose restriction to $P\times \{0\}$ is constant. Suppose now that we have a microflexible solution for the diagram
\[\xymatrix{
P\times \{0\}\ar[r]^c\ar[d] & \Pr(U')\ar[d]^r \\
P\times [0,1]\ar[r]^g & \Pr(U)}\]
where $c$ is the constant map with value $W'$. This means that we can find an $\epsilon>0$ and an open subset $C\subset V\subset U$ 
%pair of open subsets $C\subset V\subset U$ and $C'\subset V'\subset U'$ such that $V\subset V'$ and
such that
\[\xymatrix{
P\times \{0\}\ar[r]^c \ar[d]& \Pr(U)\ar[r] & \Pr(V) \ar[d]  \\
P\times [0,\epsilon)\ar[r]^g\ar[urr]^h & \Pr(U')\ar[r] & \Pr(U').
}\]
Then, for small values of $\delta\in [0,\epsilon)$, the map $h$ takes values in the space of sections of $C_k(NW)$ (the $k$ is determined because $h$ is extending the section $g$ that takes values in $C_k(NW)$, and because the inclusion $W\subset W'$ induces an epimorphism in components), and therefore defines for each $p\in P$ a $k$-sheeted covering of $W$. As a consequence, the above solution gives also a lift to the following diagram (where $\mathrm{Cov}(W)$ denotes the space of finite sheeted coverings of $W$):
\[\xymatrix{
 & \mathrm{Cov}(W') \ar[d]  \\
P\ar[r]^g\ar[ur]^h & \mathrm{Cov}(W).
}\]
But this would mean that any family of coverings of $W$ can be extended to a family of coverings of $W'$ which is false (for instance, if $W' = S^2$ and $W$ is a equatorial band in $S^2$).
\end{proof}

\bibliographystyle{amsalpha}
\bibliography{biblio-article}

\end{document}

%% file: defs.tex
\usepackage{amsmath,amstext,amssymb,mathrsfs,amscd,amsthm}
\usepackage{amsfonts}
\usepackage[all]{xy}
\usepackage{booktabs}
\usepackage{verbatim}
\usepackage{enumerate}
\usepackage[bookmarks=false]{hyperref}

\usepackage{graphicx}                   % gr{\`a}fics
\usepackage{subfig}
\usepackage{xspace}

\newtheorem{lemma}{Lemma}[section]
\newtheorem{proposition}[lemma]{Proposition}
\newtheorem*{theorem}{Theorem}
\newtheorem*{cor0}{Corollary}

\newtheorem{predf}[lemma]{Definition} 
\newenvironment{df}{\begin{predf}\rm}{\end{predf}}
\newtheorem{preremark}[lemma]{Remark}  
\newenvironment{remark}{\begin{preremark}\rm}{\end{preremark}}
\newtheorem{preremark0}[lemma]{Remark}  

\newtheorem*{prenotation}{Notation}

\numberwithin{equation}{section}

\newcommand\mnote[1]{}

\newcommand\mmnote[1]{}

\newcommand{\bN}{\mathbb{N}}

\newcommand{\bR}{\mathbb{R}}

\newcommand{\bZ}{\mathbb{Z}}

\newcommand\lra{\longrightarrow}

\newcommand\Diff{\mathrm{Diff}}
\newcommand\Emb{\mathrm{Emb}}

\newcommand\Th{\mathrm{Th}}

\newcommand{\map}{\mathrm{map}}

\renewcommand{\geq}{\geqslant}
\renewcommand{\leq}{\leqslant}

\newcommand{\cL}{\mathcal{L}}